\documentclass[12pt,reqno]{amsart}
\usepackage{amsmath,amssymb,enumerate}
\usepackage[all]{xy}
\usepackage{bm}

\newtheorem{tm}{Theorem}[section]
\newtheorem{lm}[tm]{Lemma}

\newtheorem{re}[tm]{Remark}
\newtheorem{df}[tm]{Definition}
\newtheorem{pr}[tm]{Proposition}

\makeatletter
\newcommand{\subscripts}[3]{%
  \@mathmeasure\z@\displaystyle{#2}%
  \global\setbox\@ne\vbox to\ht\z@{}\dp\@ne\dp\z@
  \setbox\tw@\box\@ne
  \@mathmeasure4\displaystyle{\copy\tw@_{#1}}%
  \@mathmeasure6\displaystyle{{#2}_{#3}}%
  \dimen@-\wd6 \advance\dimen@\wd4 \advance\dimen@\wd\z@
  \hbox to\dimen@{}\mathop{\kern-\dimen@\box4\box6}%
}
\makeatother

\newcommand{\nn}{\nonumber}

\newcommand{\LL}{\mathcal{L}}

\newcommand{\h}{\mathrm{H}}
\newcommand{\e}{\mathrm{e}}

\newcommand{\del}{\partial}
\newcommand{\ol}{\overline}

\newcommand{\la}{\langle}

\newcommand{\ra}{\rangle}

\newcommand{\dd}{\mathrm{d}}

\allowdisplaybreaks[4]

\makeatletter
 
 \@addtoreset{equation}{section}
\makeatother

\begin{document}
\title[Rate of convergence in Trotter's theorem]{Rate of convergence in Trotter's 
approximation theorem and its applications}
\author[Ryuya Namba]{Ryuya Namba}
\date{\today}
\address{Department of Mathematics,
Faculty of Education,
Shizuoka University, 836, Ohya, Suruga-ku, Shizuoka, 422-8529, Japan}
\email{{\tt{namba.ryuya@shizuoka.ac.jp}}}
\subjclass[2010]{41A25, 41A36, 60F05}
\keywords{Trotter's approximation theorem; 
operator semigroup; Berry-Esseen type estimate}

\begin{abstract}
The celebrated Trotter approximation theorem provides a 
 sufficient condition 
for the convergence of a sequence 
of operator semigroups in terms of the corresponding 
sequence of infinitesimal generators. There exist a few results 
on the rate of convergence in Trotter's theorem under some constraints. 
In the present paper,  a new rate of convergence in Trotter's theorem 
in full generality is given. 
Moreover, we see that this rate of convergence works well to
obtain quantitative estimates for some limit theorems in probability theory. 
 \end{abstract}

\maketitle 
%
%


\section{{\bf Introduction and main results}}

There has been a number of interests in approximation theory 
for semigroups of linear operators on Banach spaces among 
several branches of mathematics such as functional analysis, 
partial differential equations, probability theory and so on. 
Trotter provided a remarkably useful sufficient condition 
for the convergence of a sequence 
of operator semigroups in terms of the corresponding 
sequence of infinitesimal generators in \cite{Trotter}. 
Afterwards, several extensions of Trotter's approximation theorem 
have been discussed 
by noting some relations among operator semigroups, resolvents and  generators. 
We refer to e.g., \cite{Kurtz, Kisynski} for related early works 
and \cite{Pazy, Kato, EN} for good textbooks with 
extensive references therein. 

We now recall a general statement of Trotter's approximation theorem according to \cite{Kurtz}. 
In the following, we denote by $\|\mathfrak{A}\|$ the usual operator norm of a 
bounded linear operator $\mathfrak{A}$ defined on some Banach space.  
Let $(B_n, \|\cdot\|_{B_n}), \, n \in \mathbb{N}$, and $(E, \|\cdot\|_E)$ be 
 Banach spaces. Let $P_n : E \to B_n, \, n \in \mathbb{N},$ be a bounded linear operator
 with $\|P_n\| \le 1$. 

\begin{df}
We say that the sequence of pairs $\{(B_n, P_n)\}_{n=1}^\infty$ approximates 
the Banach space $E$ 
if  
$\|P_n f\|_{B_n} \to \|f\|_E$ as $n \to \infty$
 for every $f \in E$.
 \end{df}
The definition above means that each $P_n, \, n \in \mathbb{N},$ 
is regarded as an isomorphism between $B_n$ and $E$ when passing 
to the limit in some sense. Therefore, $P_n$ is occasionally called 
an {\it approximating operator} of $E$. 
Let $f_n \in B_n$ and $f \in E$. 
We also say that 
$f=\lim_{n \to \infty}f_n$
if 
$$
\|f_n-P_nf\|_{B_n} \to 0 
$$ 
as $n \to \infty$. 
We then define 
the limit $\mathfrak{A}$ of a sequence of 
linear operators $\mathfrak{A}_n$ with the domain $\mathrm{Dom}(\mathfrak{A}_n)$
and the range $\mathrm{Ran}(\mathfrak{A}_n)$ in $B_n$ by putting 
$$
\mathfrak{A}f:=\lim_{n \to \infty}\mathfrak{A}_nP_n f
$$
for all $f \in E$, for which this limit exists. We put 
$$
\mathrm{Dom}(\mathfrak{A})=\{f \in E \, | \, \text{there exists }
\lim_{n \to \infty}\mathfrak{A}_nP_n f\}. 
$$
Then we have the following.

\begin{pr}[cf.\,{\cite[Theorem 2.13]{Kurtz}}]
\label{Trotter's approximation}
Let $T_n, \, n \in \mathbb{N},$ be a bounded linear operator on $B_n$ with $\|T_n\| \le 1$.
Let $\{\ell(n)\}_{n=1}^\infty$ be a sequence of positive numbers and 
$\mathfrak{A}_n:=(T_n-I)/\ell(n)$ for $n \in \mathbb{N}.$ 
Suppose that $\ell(n) \to 0$ as $n \to \infty$ and that $\mathfrak{A}$ is defined by the closure of 
the limit $\lim_{n \to \infty}\mathfrak{A}_n$. 
If the domain $\mathrm{Dom}(\mathfrak{A})$ is dense in $E$ and 
the range $\mathrm{Ran}(\lambda-\mathfrak{A})$ is dense in $E$ for some $\lambda>0$, 
then there exists a $C_0$-semigroup $(\mathcal{T}_t)_{t \ge 0}$ on $E$ such that
$$
\lim_{n \to \infty}\|T_n^{[t/\ell(n)]}P_n f - P_n \mathcal{T}_t f\|_{B_n}=0, \qquad t \ge 0, \, f \in E.
$$
\end{pr}
Note that the contractivity $\|T_n\| \le 1, \, n \in \mathbb{N}$, is imposed   
for a convenience. 
Indeed, Proposition \ref{Trotter's approximation} 
can be stated under slightly weaker assumptions. 
See also \eqref{cond1} in Theorem \ref{Trotter-refine} below.

We should emphasize that Trotter's approximation theorem itself 
did not provide any quantitative estimates 
for the convergences of semigroups. 
To obtain such estimates should be 
one of the main problems of interest in a number of 
parts of approximation theory. 
So some authors have tried to consider this problem. 
As far as we know, Mangino and Rasa gave the first result 
on the rate of convergence 
in Trotter's approximation theorem in \cite{MR}. 
Moreover, Campiti and Tacelli also
established a refinement of 
Trotter's approximation theorem in \cite[Theorem 1.1]{CT}  
under a special condition $B_n \equiv E$ for $n \in \mathbb{N}$.
On the other hand, the assumption for the linear operator 
$T_n, \, n \in \mathbb{N}$, imposed in \cite{CT}
 was not sufficient in general. 
 Therefore, they wrote an additional paper \cite{CTa}, where 
 the result has already improved properly
 and an application to Bernstein operators has been given. 
 See also \cite{CTb} for a related result on the rate of convergence
 in Trotter's theorem. 
 However, we note that the cases 
 where each approximating Banach space $B_n$
 differs for every $n \in \mathbb{N}$ are still left, though
they should have a number of applications of this rate of convergence
to very extensive areas of mathematics. 

Inspired by these circumstances,  we obtain the following rate of convergence, 
which corresponds to a refinement of Proposition \ref{Trotter's approximation}
and is also regarded as a certain extension of \cite[Theorem 1.1]{CT} 
to considerable cases. 

\begin{tm}
\label{Trotter-refine}
Let $B_n, \, n \in \mathbb{N},$ be a Banach space endowed with $\|\cdot\|_{B_n}$ 
and $P_n : E \to B_n, \, n \in \mathbb{N}$, a bounded linear operator with $\|P_n\| \le 1$. 
Suppose that 
$\{(B_n, P_n)\}_{n=1}^\infty$ approximates a Banach space $E$. 
Let $T_n, \, n \in \mathbb{N},$ be a bounded linear operator on $B_n$ satisfying 
\begin{equation}\label{cond1}
\|T_n^k\| \le Me^{\omega k/n}, \qquad n, k \in \mathbb{N},
\end{equation}
for some $M \ge 1$ and $\omega \ge 0$ independent of $k$ and $n$. 
Suppose that $D$ is a dense subspace of $E$ 
and $\mathfrak{A} : (D \subset) \mathrm{Dom}(\mathfrak{A}) \to E$ is a linear operator. 
If $\mathrm{Ran}(\lambda-\mathfrak{A})$ is dense in $E$ for some $\lambda>\omega$,
then the closure of $(\mathfrak{A}, D)$ generates a $C_0$-semigroup
$(\mathcal{T}_t)_{t \geq 0}$ on $E$ satisfying $\|\mathcal{T}_t\| \le Me^{\omega t}$
for $t \ge 0$. 
Moreover, suppose that  
\begin{equation}\label{CT1}
\|n(T_n-I)P_nf\|_{B_n} \leq \varphi_n(f), \qquad f \in D,
\end{equation} 
and the following Voronovskaja-type formula holds:
\begin{equation}\label{CT2}
\|n(T_n -I)P_n f - P_n\mathfrak{A} f\|_{B_n} \leq \psi_n(f), \qquad f \in D,
\end{equation} 
where $\varphi_n : D \to [0, \infty)$ and 
$\psi_n :  D \to [0, \infty)$
are semi-norms on $D$ with $\lim_{n \to \infty}\psi_n(f)=0$ for 
$f \in D$. 
Then, for every $t \geq 0$ and for every increasing $\{k(n)\}_{n=1}^\infty$ 
of non-negative integers, we have
\begin{align}\label{rate of convergence}
\|T_n^{k(n)}P_nf - P_n\mathcal{T}_t f\|_{B_n} 
&\le M\exp(2\omega e^{\omega/n}k(n)/n)
\Big(\frac{\omega}{n}\frac{k(n)}{n} +\frac{\sqrt{k(n)}}{n} \Big)\varphi_n(f)\nn\\
&\hspace{1cm}+M\exp(\omega t_ne^{\omega/n})\Big|\frac{k(n)}{n}-t\Big|\varphi_n(f)\nn\\
&\hspace{1cm}+M\exp(\omega te^{\omega/n})\int_0^t 
\exp(-\omega se^{\omega/n})\psi_n(\mathcal{T}_s f) \, ds
\end{align}
for all $f \in D_0:=\{g \in D \, | \, \mathcal{T}_tg \in D, \, t \ge 0\}$, 
where we put $t_n:=\max\{t, k(n)/n\}$. 
\end{tm}

As is seen, the estimates \eqref{CT1} and \eqref{CT2} 
play important roles when we obtain  \eqref{rate of convergence}. 
The condition \eqref{CT1} corresponds to an estimate of 
the operator norm of the infinitesimal generator of a discrete semigroup itself. 
On the other hand, the condition
\eqref{CT2} indicates the estimate of the norm of the difference between  
the discrete infinitesimal generator and the limiting one, 
which should converge to zero as $n \to \infty$ 
by virtue of Proposition \ref{Trotter's approximation}.

The most typical choice of the sequence $\{k(n)\}_{n=1}^\infty$ is that 
$k(n):=[nt]$ for $n \in \mathbb{N}$ and $t \ge 0$. 
Since it holds that $k(n)=[nt] \le nt$, $t_n=t$ and $|[nt]/n - t| \le 1/n$, 
Inequality \eqref{rate of convergence}  becomes 
$$
\begin{aligned} 
\|T_n^{[nt]}P_nf - P_n\mathcal{T}_t f\|_{B_n} 
&\le M\exp(2\omega te^{\omega/n})
\Big(\frac{\omega t}{n}  +\sqrt{\frac{t}{n}} \Big)\varphi_n(f)
+\frac{M}{n}\exp(\omega t e^{\omega/n})\varphi_n(f)\nn\\
&\hspace{1cm}+M\exp(\omega te^{\omega/n})\int_0^t 
\exp(-\omega se^{\omega/n})\psi_n(\mathcal{T}_s f) \, ds
\end{aligned}
$$
for $f \in D_0$ and $t \ge 0$. 
Moreover, if we can especially take
$M=1$ and $\omega=0$, then the estimate above can be also written as 
the following:
\begin{align} 
&\|T_n^{[nt]}P_nf - P_n\mathcal{T}_t f\|_{B_n} \le    \sqrt{\frac{t}{n}}  \varphi_n(f)
+\frac{1}{n} \varphi_n(f) 
+\int_0^t \psi_ n(\mathcal{T}_sf) \, ds
\label{Trotter-refine-2}
\end{align}
for $f \in D_0$
and $t \ge 0$.

In Section 2, we give the proof of Theorem \ref{Trotter-refine}. 
Since the Banach space $B_n$ on which the operator $T_n$ is defined 
may vary as $n$ does, one may think of the proof as 
more complicated than that  given in \cite{CT}. 
However, we can see that such generality does not essentially affect the proof itself, 
which implies that we can obtain many kinds of rates of convergences under 
various settings. Moreover, we also give a rate of convergence 
under an additional assumption on  the limiting semigroup $(\mathcal{T}_t)_{t \ge 0}$
as an immediate consequence of the main result. 
Section 3 is devoted to applications of Theorem \ref{Trotter-refine} 
to the rates of convergences for central limit theorems (CLTs, in short)
in probability theory. The speed rate of the CLT is called the 
Berry--Esseen type bound and it corresponds to a certain refinement of the CLT. 
So far, a lot of ways to establish this kind of bound are known. 
See e.g., \cite[Chapter XVI]{Feller} for the proof
of the Berry--Esseen type bound based on the convergence 
of characteristic functions. 
On the other hand, there is an alternative representation of the CLT
in terms of the convergence of semigroups, whose proof is given by employing
Trotter's approximation theorem. 
We give the Berry--Esseen type bound for the semigroup CLT by using 
Theorem \ref{Trotter-refine}. As a further problem, we also consider a CLT
for magnetic transition operators on crystal lattices discussed in \cite{Kotani}.
We give a quantitative estimate of the CLT by applying Theorem \ref{Trotter-refine}
as well.


\section{{\bf Proof of Theorem \ref{Trotter-refine}}}
 
The proof of Theorem \ref{Trotter-refine} is given in this section. 
We basically follow the argument given in \cite[Theorem 1.1]{CT}.  
Here, we should pay an attention to the proof 
since the Banach space $B_n$ may vary for each $n \in \mathbb{N}$ in our setting.  
However, as we can see below, somewhat surprisingly, such general settings
do not essentially affect the proof.  
Moreover, it is pointed out in \cite{CTa} that some commutativeness conditions like 
$T_n T_m=T_m T_n, \, m, n \in \mathbb{N}$, play an important role in the proof. 
On the other hand, in our setting, such a condition does not make sense 
since the domains of $T_n$ and $T_m$ may be distinct. 
A key to our proof is a condition that not only $f \in D$ but also $\mathcal{T}_tf \in D$,
$t \ge 0$, holds, which is introduced as an alternative condition in \cite{CTa}.

\begin{proof}[Proof of Theorem \ref{Trotter-refine}]
The existence of the $C_0$-semigroup $(\mathcal{T}_t)_{t \geq 0}$ generated by 
the closure of $(\mathfrak{A}, D)$ has been showed in \cite[Theorem 2.13]{Kurtz}.
Therefore, we concentrate on the proof of  \eqref{rate of convergence}. 
We split the proof into four steps. 

\vspace{2mm}
\noindent
{\bf Step 1}. 
Consider the bounded linear operator $\mathfrak{A}_n=n(T_n-I)$ on $B_n$ 
for $n \in \mathbb{N}$, which generates a $C_0$-semigroup $(S_t^{(n)})_{t \ge 0}$ on $B_n$
given by 
$$
S_t^{(n)}=\e^{t\mathfrak{A}_n}=e^{-nt}\e^{ntT_n}=
e^{-nt}\Big(
\sum_{k=0}^\infty \frac{(nt)^k}{k!}T_n^k\Big).
$$
Note that \eqref{cond1} implies
\begin{equation}
\label{S-est}
\|S_t^{(n)}\| \le 
e^{-nt}\Big(
\sum_{k=0}^\infty \frac{(nt)^k}{k!}\|T_n^k\|\Big) 
\le M\exp\big(nt(e^{\omega/n}-1)\big), 
\qquad n \in \mathbb{N}, \, t \ge 0. 
\end{equation}
Recall that $D_0=\{g \in D \, | \, \mathcal{T}_tg \in D, \, t \ge 0\}$. 
Let $\{k(n)\}_{n=1}^\infty$ be an increasing 
sequence of positive integers and 
$f \in D_0$. We then have
\begin{align}
&\|T_n^{k(n)}P_nf - P_n\mathcal{T}_tf\|_{B_n} \nn\\
&\le \|T_n^{k(n)}P_nf - S_{k(n)/n}^{(n)}P_nf\|_{B_n} \nn\\
&\hspace{1cm}+\|S_{k(n)/n}^{(n)}P_nf - S_t^{(n)}P_n f\|_{B_n}
+\|S_t^{(n)}P_nf - P_n\mathcal{T}_t f\|_{B_n}\nn\\
&=:I_1(n)+I_2(n)+I_3(n). 
\label{est1}
\end{align}
We now try to estimate each term on the right-hand side of \eqref{est1}.

\vspace{2mm}
\noindent
{\bf Step 2}. We here give an estimation of $I_1(n)$. 
By applying \cite[Lemma III.5.1]{Pazy} and an elementary inequality
$e^x-1 \le xe^x$ for $x \ge 0$, we obtain
\begin{align}
&\|T_n^{k(n)}P_nf - S_{k(n)/n}^{(n)}P_nf\|_{B_n} 
=\|\e^{k(n)(T_n-I)}P_n f - T_n^{k(n)}P_n f\|_{B_n} \nn\\
&\le M\exp\big(\omega(k(n)-1)/n\big)
\exp\big(k(n)(e^{\omega/n}-1)\big)\nn\\
&\hspace{1cm}\times\sqrt{k(n)^2(e^{\omega/n}-1)^2+k(n)e^{\omega/n}}\|(T_n-I)P_n f\|_{B_n} \nn \\
&\le \frac{M}{n}\exp\big(\omega(k(n)-1)/n\big)\exp\big(k(n)(e^{\omega/n}-1)\big)\nn\\
&\hspace{1cm}\times\Big( k(n)(e^{\omega/n}-1)+\sqrt{k(n)e^{\omega/n}}\Big)\varphi_n(f) \nn\\
&\le M\exp\big(\omega(k(n)-1)/n\big)\exp\big(e^{\omega/n}\omega k(n)/n\big)
\Big(\frac{\omega}{n}\frac{k(n)}{n}e^{\omega/n}+\frac{\sqrt{k(n)}}{n}e^{\omega/2n}\Big)\varphi_n(f) \nn\\
&= M\exp\big(\omega(e^{\omega/n}+1)k(n)/n\big)e^{-\omega/n}
\Big(\frac{\omega}{n}\frac{k(n)}{n}e^{\omega/n}+\frac{\sqrt{k(n)}}{n}e^{\omega/2n}\Big)\varphi_n(f) \nn \\
&\le M\exp\big(2\omega e^{\omega/n}k(n)/n\big)
\Big(\frac{\omega}{n}\frac{k(n)}{n} +\frac{\sqrt{k(n)}}{n} \Big)\varphi_n(f). 
\label{goal1}
\end{align}
Moreover, the estimation of $I_2(n)$ in \eqref{est1} is given by
\begin{align}
&\|S_{k(n)/n}^{(n)}P_nf - S_t^{(n)}P_n f\|_{B_n} \nn\\
&=\Big\|\int_t^{k(n)/n} S_s^{(n)}\big(n(T_n-I)\big)P_nf \, ds\Big\|_{B_n} \nn\\
&\le M\exp\big(nt_n(e^{\omega/n}-1)\big)\Big|\frac{k(n)}{n}-t\Big|\varphi_n(f)\nn\\
& \le M\exp\big(\omega t_ne^{\omega/n}\big)\Big|\frac{k(n)}{n}-t\Big|\varphi_n(f), 
 \label{goal2}
\end{align} 
where we recall that $t_n:=\max\{t, k(n)/n\}$.

\vspace{2mm}
\noindent
{\bf Step 3.} 
Since both operators $S_{t-s}^{(n)} : B_n \to B_n$ and $P_n \mathcal{T}_t : E \to B_n$
are bounded, we have 
$$
\begin{aligned}
S_t^{(n)}P_n f - P_n \mathcal{T}_t f
&=-\int_0^t \frac{d}{ds}(S_{t-s}^{(n)}P_n \mathcal{T}_s)f \, ds \\
&=\int_0^t (S_{t-s}^{(n)}\mathfrak{A}_n P_n \mathcal{T}_s - 
S_{t-s}^{(n)}P_n \mathcal{T}_s \mathfrak{A}) f \, ds \\
&=\int_0^t S_{t-s}^{(n)}(\mathfrak{A}_n P_n - 
P_n \mathfrak{A})\mathcal{T}_s  f \, ds
\end{aligned}
$$
for $t \ge 0$, where we used the fact that $\mathcal{T}_s$ and its generator $\mathfrak{A}$ commute. 
Therefore, it follows from \eqref{S-est}, \eqref{CT2} and $\mathcal{T}_sf \in D$ that 
\begin{align}
&\|S_t^{(n)}P_n f - P_n \mathcal{T}_t f\|_{B_n}\nn\\
&=\Big\|\int_0^t S_{t-s}^{(n)}(\mathfrak{A}_n P_n - 
P_n \mathfrak{A})\mathcal{T}_s  f \, ds\Big\|_{B_n} \nn\\
&\le M\int_0^t \exp\big(n(t-s)(e^{\omega/n}-1)\big)
\|(\mathfrak{A}_n P_n - 
P_n \mathfrak{A})\mathcal{T}_s  f\|_{B_n} \, ds \nn\\
&\le M\int_0^t \exp\big(\omega(t-s)e^{\omega/n}\big)\psi_n(\mathcal{T}_s f) \, ds\nn\\
&\le M\exp\big(\omega te^{\omega/n}\big)\int_0^t \exp\big(-\omega se^{\omega/n}\big)
\psi_n(\mathcal{T}_s f) \, ds
 \label{goal3}
\end{align}
for $t \ge 0$.

\vspace{2mm}
\noindent
{\bf Step 4}. We combine \eqref{est1} with 
\eqref{goal1}, \eqref{goal2} and \eqref{goal3}. Then we obtain
$$
\begin{aligned}
&\|T_n^{k(n)}P_nf - P_n\mathcal{T}_tf\|_{B_n} \\
&\le M\exp\big(2\omega e^{\omega/n}k(n)/n\big)
\Big(\frac{\omega}{n}\frac{k(n)}{n} +\frac{\sqrt{k(n)}}{n} \Big)\varphi_n(f)\\
&\hspace{1cm}+M\exp\big(\omega t_ne^{\omega/n}\big)\Big|\frac{k(n)}{n}-t\Big|\varphi_n(f)\\
&\hspace{1cm}+M\exp\big(\omega te^{\omega/n}\big)
\int_0^t \exp\big(-\omega se^{\omega/n}\big)\psi_n(\mathcal{T}_s f) \, ds
\end{aligned}
$$
for all $f \in D_0$ and $t \ge 0$,
which is the very desired estimate \eqref{rate of convergence}.  
\end{proof}

Before closing this section, 
we give a corollary of Theorem \ref{Trotter-refine}, under 
an additional assumption that 
the limiting semigroup $(\mathcal{T}_t)_{t \ge 0}$ 
preserves $D$ and the seminorm $\psi_n$ and the limiting semigroup
are commutative in a sense.

\begin{re}\label{Remark}
We assume that $\mathcal{T}_t(D) \subset D$ for $t \ge 0$ and 
$\psi_n(\mathcal{T}_t f) \le \|\mathcal{T}_t\| \psi_n(f)$ for 
$f \in D$ and $t \ge 0$. 
Then we have 
$$
\int_0^t \exp\big(-\omega se^{\omega/n}\big)\psi_n(\mathcal{T}_s f) \, ds \le 
\int_0^t \exp\big(\omega s(1-e^{\omega/n})\big)\psi_n(f) \, ds
\le t \psi_n(f)
$$
for $f \in D$ and $t \ge 0$. 
Therefore, Formula \eqref{rate of convergence} becomes
\begin{align}
&\|T_n^{k(n)}P_nf - P_n\mathcal{T}_t f\|_{B_n}\nn\\
&\le M\exp\big(2\omega e^{\omega/n}k(n)/n\big)
\Big(\frac{\omega}{n}\frac{k(n)}{n} +\frac{\sqrt{k(n)}}{n} \Big)\varphi_n(f)\nn\\
&\hspace{1cm}+M\exp\big(\omega t_ne^{\omega/n}\big)\Big|\frac{k(n)}{n}-t\Big|\varphi_n(f)
+Mt\exp\big(\omega t e^{\omega/n}\big) \psi_n(f)
  \label{Final}
\end{align}
for $f \in D$ and $t \ge 0$. 
\end{re}

\vspace{2mm}


\section{{\bf Applications of Theorem \ref{Trotter-refine}}}

This section is concerned with several applications of the rate of convergence in 
Trotter's approximation theorem to obtain some quantitative estimates for
limit theorems in probability theory. 

\subsection{Quantitative estimates of CLTs}
It is known that the CLT plays a crucial role in probability theory. 
Let $\{\xi_i\}_{i=1}^\infty$ be a sequence of independently and 
identically distributed (i.i.d., in short) 
$\mathbb{Z}^d$-valued random variables given by
$$
\mathbb{P}(\xi_1=\bm{e}_k)=\mathbb{P}(\xi_1=-\bm{e}_k)=\frac{1}{2d}, 
\qquad k=1, 2, \dots, d,
$$
where $\bm{e}_k=(0, \dots, 0, \overbrace{1}^{k\text{th}}, 0, \dots, 0) \in \mathbb{Z}^d$ 
is the unit vector for $k=1, 2, \dots, d$. 

The CLT describes the fluctuation of the random variable defined by
$$
X_n:=\frac{\xi_1+\xi_2+\cdots+\xi_n}{\sqrt{n}},\qquad n \in \mathbb{N},
$$
as $n$ tends to infinity. More precisely, it asserts the convergence 
of the distribution of $X_n$ to the 
$d$-dimensional standard normal distribution $N(\bm{0}, I)$ as $n \to \infty$, 
where $I$ denotes the $d \times d$-identity matrix. 
Note that another representation of the CLT is given in terms of
the convergence of the discrete semigroups associated with $X_n$ 
to the continuous heat semigroup generated by the Laplacian on $\mathbb{R}^d$. 
Note that the assertion above is easily extended 
to the case where
the sequence of $\mathbb{R}^d$-valued i.i.d.~random variables 
$\{\xi_i=(\xi_i^1, \xi_i^2, \dots, \xi_i^d)\}_{i=1}^\infty$ satisfies 
$\mathbb{E}[\xi_1]=\mu \in \mathbb{R}^d$ and 
$\mathrm{Cov}(\xi_1^i, \xi_1^j)=\sigma_{ij}$ 
so that $(\sigma_{ij})_{i, j=1}^d$ forms a positive semidefinite
symmetric matrix, though we need to replace $X_n$ by
$$
X_n=\frac{\xi_1+\xi_2+\cdots+\xi_n - n\mu}{\sqrt{n}},\qquad n \in \mathbb{N}.
$$

As a refinement of the CLT, the {\it Berry--Esseen type bound} 
is well-known, which gives a rate of convergence of the CLT in the parameter $n$. 
We see that the Berry--Esseen type bound is easily obtained by a simple application of Theorem \ref{Trotter-refine}. 
Let us put $B_n \equiv C_\infty(\mathbb{Z}^d)$ for $n \in \mathbb{N}$ 
endowed with the sup-norm $\|\cdot\|_\infty$
and $E=C_\infty(\mathbb{R}^d)$ with $\|\cdot\|_\infty$. Here, 
we denote by 
$C_\infty(M)$ the space of all functions on a topological space $M$
vanishing at infinity. We define a bounded linear operator 
$P_n : C_\infty(\mathbb{R}^d) \to C_\infty(\mathbb{Z}^d), \, n \in \mathbb{N}$, by
 $$ 
 P_nf(x):=f(n^{-1/2}x),\qquad x \in \mathbb{Z}^d.
 $$
 Then we easily see that $\|P_n\| \le 1, \, n \in \mathbb{N}$, and 
 the sequence $\{(C_\infty(\mathbb{Z}^d), P_n)\}_{n=1}^\infty$ approximates the 
 Banach space $(C_\infty(\mathbb{R}^d), \|\cdot\|_\infty)$. 

We put $\mathcal{E}:=\{\pm \bm{e}_1, \pm \bm{e}_2, \dots, \pm\bm{e}_d\}$
and define a linear operator $T_n, \, n \in \mathbb{N},$ on 
$C_\infty(\mathbb{Z}^d)$ by 
$$
T_n f(x) \equiv \mathcal{L} f(x):=\frac{1}{2d}\sum_{\bm{e} \in \mathcal{E}} f(x+\bm{e}), \qquad x \in \mathbb{Z}^d. 
$$
The operator $\mathcal{L}$ is called the {\it transition operator} 
associated with $\{\xi_i\}_{i=1}^\infty$ in the context of probability theory. 
Note that $\|\mathcal{L}\| \le 1$ holds. 

We define a subspace $D$ by 
$$
\begin{aligned}
D:=C_\infty^\infty(\mathbb{R}^d) &= \bigcap_{k=1}^\infty \Big\{ f \in C_\infty(\mathbb{R}^d) \, : \, 
\lim_{|x| \to \infty} \frac{\partial^k f}{\partial x_{i_1} \partial x_{i_2}\cdots \partial x_{i_k}}(x)=0, \\ 
& \hspace{1cm} i_1, i_2, \dots, i_k \in \{1, 2, \dots, d\}\Big\},
\end{aligned}
$$ 
the set of $C^\infty$-functions on $\mathbb{R}^d$ all of whose  
partial derivatives of an arbitrary order vanish at infinity. 
We easily see that $D$ is a dense subspace of $C_\infty(\mathbb{R}^d)$. 
Moreover, it is known that $\mathrm{Ran}(\lambda - \Delta)$ is dense in 
$C_\infty(\mathbb{R}^d)$ for some $\lambda>0$ and 
the closure of $(\mathfrak{A}=\Delta, C_\infty^\infty(\mathbb{R}^d))$ 
generates a heat semigroup $(\mathcal{T}_t=\e^{t\Delta})_{t \ge 0}$ 
(see e.g., \cite[Proposition 4.4.4]{Kal}). 
Here $\Delta=\sum_{i=1}^d (\del^2/\del x_i^2)$ stands for the (negative) 
Laplacian
on $\mathbb{R}^d$. 
Under these settings, 
the CLT can be also written as follows: 
\begin{equation}\label{CLT}
\lim_{n \to \infty}\|\mathcal{L}^{[nt]} P_n f - P_n \e^{t\Delta}f\|_\infty=0, 
\qquad f \in C_\infty^\infty(\mathbb{R}^d), \, t \ge 0. 
\end{equation}

We can show that
\begin{equation}\label{phi1}
\|n(\LL - I)P_n f\|_\infty \le \varphi_n(f) := \|\Delta f\|_\infty
+\frac{d}{6\sqrt{n}}\max_{i=1, 2, \dots, d}\Big\|\frac{\del^3 f}{\del x_i^3}
\Big\|_\infty
\end{equation}
and 
\begin{equation}\label{psi1}
\|n(\LL - I)P_n f - P_n \Delta f\|_\infty \le \psi_n(f):=
\frac{d}{6\sqrt{n}}\max_{i=1, 2, \dots, d}\Big\|\frac{\del^3 f}{\del x_i^3}
\Big\|_\infty
\end{equation}
for $f \in C_\infty^\infty(\mathbb{R}^d)$. 
Indeed, by applying the Taylor formula to the function $f$ at $x/\sqrt{n}$, we have 
$$
\begin{aligned}
&n(\LL - I)P_n f(x) \\
&= \frac{n}{2d}\sum_{\bm{e} \in \mathcal{E}} 
f\Big(\frac{x+\bm{e}}{\sqrt{n}}\Big)  - nf\Big(\frac{x}{\sqrt{n}}\Big)\\
&=\frac{1}{2d}\sum_{\bm{e} \in \mathcal{E}} 
\Bigg\{ \sqrt{n}\sum_{i=1}^d \frac{\del f}{\del x_i}\Big(\frac{x}{\sqrt{n}}\Big)
\bm{e}^i
+\frac{1}{2}\sum_{i, j=1}^d \frac{\del^2 f}{\del x_i \del x_j}\Big(\frac{x}{\sqrt{n}}\Big)
\bm{e}^i \bm{e}^j \\
&\hspace{1cm}+\frac{1}{6\sqrt{n}}\sum_{i, j, k=1}^d 
\frac{\del^3 f}{\del x_i \del x_j \del x_k}(\theta)\bm{e}^i\bm{e}^j\bm{e}^k\Bigg\}\\
&=\frac{1}{2d}\sum_{\bm{e} \in \mathcal{E}} 
\Bigg\{ \sqrt{n}\sum_{i=1}^d \frac{\del f}{\del x_i}\Big(\frac{x}{\sqrt{n}}\Big)
\bm{e}^i
+\frac{1}{2}\sum_{i=1}^d \frac{\del^2 f}{\del x_i^2}\Big(\frac{x}{\sqrt{n}}\Big)
(\bm{e}^i)^2
+\frac{1}{6\sqrt{n}}\sum_{i=1}^d 
\frac{\del^3 f}{\del x_i^3}(\theta)(\bm{e}^i)^3\Bigg\}
\end{aligned}
$$
for any $f \in C_\infty^\infty(\mathbb{R}^d)$ and  some 
$\theta=\theta(\bm{e}) \in \mathbb{R}^d$, where 
$\bm{e}^i, \, i=1, 2, \dots, d$, denotes the $i$th component of $\bm{e}$. 
By virtue of 
$$
\sum_{\bm{e} \in \mathcal{E}}\bm{e}^i=0, \quad 
\sum_{\bm{e} \in \mathcal{E}}(\bm{e}^i)^2=2, \qquad i=1, 2, \dots, d, 
$$ 
we have 
$$
n(\LL - I)P_n f(x) =P_n \Delta f(x)+\frac{1}{12d\sqrt{n}}
\sum_{e \in \mathcal{E}}\sum_{i=1}^d 
\frac{\del^3 f}{\del x_i^3}(\theta)(\bm{e}^i)^3.
$$
Hence, we conclude
$$
\begin{aligned}
\|n(\LL - I)P_n f\|_\infty &\le \| \Delta f \|_\infty 
+\frac{d}{6\sqrt{n}}\max_{i=1, 2, \dots, d}\Big\|\frac{\del^3 f}{\del x_i^3}\Big\|_\infty,\\
\|n(\LL - I)P_nf - P_n \Delta f\|_\infty &\le
\frac{d}{6\sqrt{n}}\max_{i=1, 2, \dots, d}\Big\|\frac{\del^3 f}{\del x_i^3}\Big\|_\infty
\end{aligned}
$$
for all $f \in C_\infty^\infty(\mathbb{R}^d)$,
which are the desired estimates \eqref{phi1} and \eqref{psi1}.

Since $\mathcal{T}_t(D) \subset D$ and 
$\psi_n(\mathcal{T}_tf) \le \psi_n(f)$
hold for $t \ge 0$ and $f \in D$ by definition, Theorem \ref{Trotter-refine} (in particular, 
Equation \eqref{Final} in Remark \ref{Remark}) allows us to establish the following 
refinement of \eqref{CLT}. 

\begin{tm}\label{Thm:BerryEsseen}
Suppose that  $f \in C_\infty^\infty(\mathbb{R}^d)$ and $t \ge 0$. 
Then, there exists a positive constant $C=C(t, f, d)>0$ such that
$$
\begin{aligned}
&\|\mathcal{L}^{[nt]} P_n f - P_n \e^{t\Delta}f\|_\infty\\
& \le \Big(\sqrt{\frac{t}{n}}+\frac{1}{n}\Big)\Big(
\|\Delta f\|_\infty
+\frac{d}{6\sqrt{n}}\max_{i=1, 2, \dots, d}\Big\|\frac{\del^3 f}{\del x_i^3}\Big\|_\infty\Big)
+\frac{td}{6\sqrt{n}}\max_{i=1, 2, \dots, d}\Big\|
\frac{\del^3f }{\del x_i^3}\Big\|_\infty 
\le \frac{C}{\sqrt{n}}
\end{aligned}
$$
for all $n \in \mathbb{N}$.
\end{tm}
This theorem implies that the rate of the convergence in the usual CLT 
is of order $n^{-1/2}$, which is a fundamental result in numerical calculations 
of some discrete approximation schemes of diffusion processes 
such as Brownian motions with values in $\mathbb{R}^d$. 
We should emphasize that a lot of studies to establish error bounds similar to 
Theorem \ref{Thm:BerryEsseen} are known. 
However, the Berry--Esseen type bound for the CLT in terms of semigroups 
has not appeared in existing literatures.
We refer to e.g., \cite[Chapter XVI]{Feller} for the proof
of the Berry--Esseen type bound based on the convergence 
of characteristic functions.

\subsection{Quantitative estimates of CLTs for the magnetic transition operator}

In this subsection, we give another application of Theorem \ref{Trotter-refine}
to find out the rate of convergence of CLTs for magnetic transition operators
on crystal lattices. 
Before fixing the setting, 
we briefly review the magnetic Schr\"odinger operator on $\mathbb{R}^d$. 
Let $B$ be a closed 2-form on $\mathbb{R}^d$, which is called a 
{\it magnetic field} on $\mathbb{R}^d$. 
Let $A$ be a vector potential for $B$, that is, $\dd A=B$, where $\dd$ is the exterior derivative. 
We put $\nabla_A:=\dd-\sqrt{-1}A$. 
Then the {\it magnetic Schr\"odinger operator} is given by $\nabla_A^*\nabla_A$. 
We see that the magnetic field $B$ is periodic with respect to $\mathbb{Z}^d$
if and only if $\sigma^*A-A=\dd f_\sigma, \, \sigma \in \mathbb{Z}^d$
for some $f_\sigma \in C^\infty(\mathbb{R}^d)$. 
Moreover, if it holds that 
$B=\sum_{1 \le i<j \le d}b_{ij}\dd x_i \wedge \dd x_j$ with some $b_{ij} \in \mathbb{R}$, 
we then take a linear vector potential $A=\sum_{i, j=1}^d a_{ij}x_j \dd x_i$, where 
$b_{ij}=a_{ji} - a_{ij}$ for $i, j=1, 2, \dots, d$.

A {\it crystal lattice} is defined to be a covering graph $X=(V, E)$ of a finite graph $X_0=(V_0, E_0)$
whose covering transformation group is isomorphic to $\mathbb{Z}^d$. 
Here, $V$ (resp.\,$V_0$) is the set of all vertices and $E$ (resp.\,$E_0$) 
is the set of all oriented edges of $X$ (resp.\,$X_0$). 
For an edge $e \in E$, we denote by $o(e), t(e), \ol{e}$ the origin, the terminus
and the inverse edge of $e$, respectively. We put 
$E_x:=\{e \in E \, | \, o(e)=x\}$ for $x \in V$. 
Intuitively, a crystal lattice is an infinite graph with a fundamental pattern 
consisting of finite number of edges and vertices, which appears periodically.

Let us consider a discrete analogue of the semigroup generated by
the Schr\"odinger operator with periodic magnetic field. 
Let $p : E \to (0, 1]$ be a $\mathbb{Z}^d$-invariant transition probability on $X$, 
that is, $\sum_{e \in E_x}p(e)=1$ for $x \in V$ and $p(\gamma e)=p(e)$
for $\gamma \in \mathbb{Z}^d$ and $e \in E$. 
Here, $\gamma e$  means the parallel translation of $e$ 
along $\gamma \in \mathbb{Z}^d$. 
Note that $p$ is also induced on the finite quotient graph 
$X_0=\mathbb{Z}^d \backslash X$ through the covering map $\pi : X \to X_0$. 
Then the Perron--Frobenius theorem implies the unique existence of the 
normalized invariant measure $m$ on $V_0$. Namely, $m$ is a positive function on $V_0$ satisfying 
$$
\sum_{e \in (E_0)_x}p(\ol{e})m\big(t(e)\big)=m(x), \qquad x \in V_0, \quad \text{and} \quad
\qquad \sum_{x \in V_0}m(x)=1. 
$$ 
In the present paper, we assume the {\it detailed balanced condition}
$$
p(e)m\big(o(e)\big)=p(\ol{e})m\big(t(e)\big), \qquad e \in E_0.
$$
Then the random walk induced by $p$ is said to be ($m$-){\it symmetric}.  
We define the {\it magnetic transition operator} 
$H_\omega : C_\infty(X) \to C_\infty(X)$ by 
$$
H_\omega f(x):=\sum_{e \in E_x}p(e)\exp\big(\sqrt{-1}\omega(e)\big)f\big(t(e)\big), 
\qquad x \in V,
$$
where $\omega : E \to \mathbb{R}$ is a 1-cochain on $X$ 
satisfying $\omega(\ol{e})=-\omega(e)$ for $e \in E$. 
We set the following technical but natural 
conditions for 1-cochains $\omega : E \to \mathbb{R}$. 

\vspace{2mm}
\noindent
{\bf (A1)}: 
$\omega$ is {\it weakly $\mathbb{Z}^d$-invariant}, that is,
the cohomology class $[\omega] \in \h^1(X, \mathbb{R})$
is $\mathbb{Z}^d$-invariant, where $\h^1(X, \mathbb{R})$ 
is the first cohomology group of $X$. 

\vspace{2mm}
\noindent
{\bf (A2)}: For every $\sigma \in \mathbb{Z}^d$, it holds that 
$$
\sum_{e \in E_x}p(e)\Big(\omega(\sigma^{-1}e) - \omega(e)\Big)=0, \qquad x \in V.
$$

\vspace{2mm}
\noindent
{\bf (A3)}: It holds that 
$\sigma_1 (\sigma_2\omega - \omega)=\sigma_2\omega - \omega$ for
$\sigma_1, \sigma_2 \in \mathbb{Z}^d$. 

\vspace{2mm}
Both  {\bf (A1)} and {\bf (A3)} essentially mean the invariance of 
a 1-cochain $\omega$ under the $\mathbb{Z}^d$-action. 
On the other hand, a 1-cochain satisfying {\bf (A2)} is said to be 
{\it harmonic}, which corresponds to a discrete analogue of 
a harmonic form on Riemannian manifolds. 
In fact, for $b \in \mathbb{R}$, the {\it classical Harper operator} 
on $\mathbb{Z}^2$ defined by 
$$
\begin{aligned}
&(H_bf)(m, n)\\
&:=\frac{1}{4}\Big\{\exp\Big(\frac{1}{2}\sqrt{-1}bn\Big)f(m+1, n)+
\exp\Big(-\frac{1}{2}\sqrt{-1}bn\Big)f(m-1, n)\\
&\hspace{1cm}
+\exp\Big(-\frac{1}{2}\sqrt{-1}bm\Big)f(m, n+1)+
\exp\Big(\frac{1}{2}\sqrt{-1}bn\Big)f(m, n-1)\Big\}
\end{aligned}
$$
for $(m, n) \in \mathbb{Z}^2$
satisfies {\bf (A1)}, {\bf (A2)} and {\bf (A3)}. 
Hence, the operator $H_\omega$ with these conditions is also called 
the {\it generalized Harper operator} on $X$. 

A piecewise linear map $\Phi : V \to \mathbb{R}^d$ is called a 
{\it periodic realization} of a crystal lattice $X$  if 
it satisfies $\Phi(\sigma x)=\Phi(x)+\sigma$ for $x \in V$ and $\sigma \in \mathbb{Z}^d$. 
By noting geometric features of crystal lattices,  
Kotani obtained the following CLT of semigroup type for magnetic transition operators.

\begin{pr}[cf.\,{\cite[Theorem 4]{Kotani}}]
\label{Kotani-CLT}
Let $\Phi_0 : X \to \mathbb{R}^d$ be a periodic realization of $X$ satisfying 
\begin{equation}\label{harmonic}
\sum_{e \in E_x}p(e)\Big\{\Phi_0\big(t(e)\big)-\Phi_0\big(o(e)\big)\Big\}=0,
\qquad x \in V. 
\end{equation}
Suppose that $\omega$ satisfies {\bf (A1)}, {\bf (A2)} and {\bf (A3)}. 
Then, 
there exists a flat Riemannian metric $g$ on $\mathbb{R}^d$,  
a linear vector potential $A=\sum_{i, j=1}^d a_{ij}x_j \, dx_i$ on $(\mathbb{R}^d, g)$ 
and a harmonic {\rm 1}-form $\omega_0$  on $X_0$ such that 
\begin{align}\label{omega}
\omega(e)=-\la {\bf{A}}\Phi_0\big(o(e)\big), v_e\ra_{g} - \frac{1}{2}
\la  {\bf{A}}v_e, v_e\ra_{g}+\pi^*\omega_0(e), \qquad e \in E,
\end{align}
where $v_e:=\Phi_0\big(t(e)\big)
-\Phi_0\big(o(e)\big)$ for $e \in E$ and
${\bf{A}}=(a_{ij})_{i, j=1}^d$. 
Moreover, we have
$$
\lim_{n \to \infty} \|H_{\frac{1}{n}\omega}^{[nt]} P_{n}f
- P_{n}\e^{t\nabla_A^*\nabla_A}f\|_\infty=0
$$
for every $f \in C_\infty^\infty(\mathbb{R}^d)$ and $t \ge 0$, 
where $P_n : C_\infty(\mathbb{R}^d) \to C_\infty(X)$ is an approximation operator 
given by 
$$
P_nf(x):=f\Big(\frac{1}{\sqrt{n}}\Phi_0(x)\Big), \qquad x \in V, \, n \in \mathbb{N}. 
$$

\end{pr}
We note that, if $\omega=0$, then the operator $\nabla_A^*\nabla_A$ becomes 
the usual (negative) Laplacian $\Delta$ on $(\mathbb{R}^d, g)$. 
The flat metric $g$ on $\mathbb{R}^d$ above is called the {\it Albanese metric}. 
See e.g., \cite{KS06} for its geometric meaning as well as 
its explicit construction. 

By applying Theorem \ref{Trotter-refine}, we show the following 
quantitative estimate of Proposition \ref{Kotani-CLT}. 

\begin{tm}\label{Kotani-Trotter}
For $f \in C_\infty^\infty(\mathbb{R}^d)$ and $t \ge 0$, there exists a positive constant 
$C=C(t, f, \Phi_0)>0$ such that
$$
\|H_{\frac{1}{n}\omega}^{[nt]} P_{n}f
- P_{n}\e^{t\nabla_A^*\nabla_A}f\|_\infty \le \frac{C}{\sqrt{n}}, \qquad n \in \mathbb{N}. 
$$
\end{tm}

Before giving the proof of Theorem \ref{Kotani-Trotter}, 
we show the following lemma.

\begin{lm}\label{Lemma}
Let $\Phi_0 : V \to \mathbb{R}^d$ be a periodic realization satisfying \eqref{harmonic}. 
Then there exists a positive constant $C=C(f, \Phi_0)>0$ such that 
\begin{equation}\label{Sch-1}
\|n(H_{\frac{1}{n}\omega} - I)P_n f\|_\infty \le \varphi_n(f)=
\big\|(\nabla_A^*\nabla_A)f\big\|_\infty + \frac{C}{\sqrt{n}},
\qquad f \in C_\infty^\infty(\mathbb{R}^d),
\end{equation}
and 
\begin{equation}\label{Sch-2}
\|n(H_{\frac{1}{n}\omega} - I)P_n f - P_n(\nabla_A^*\nabla_A)f\|_\infty \le \psi_n(f)=
\frac{C}{\sqrt{n}},
\qquad f \in C_\infty^\infty(\mathbb{R}^d).
\end{equation}
\end{lm}

\begin{proof}
By applying the Taylor formula to $\exp(\sqrt{-1}\omega(e)/n)$
and by noting \eqref{omega}, we have  
$$
\begin{aligned}
\exp\Big(\frac{\sqrt{-1}}{n}\omega(e)\Big)
&=1-\frac{\sqrt{-1}}{\sqrt{n}}
\Big\la {\bf{A}}\Big(\frac{1}{\sqrt{n}}\Phi_0\big(o(e)\big)\Big), v_e\Big\ra_g\\
&\hspace{1cm}-\frac{1}{2n}\Bigg(\sqrt{-1}\la {\bf{A}}v_e, v_e\ra_g+2\sqrt{-1}\pi^*\omega_0(e)\\
&\hspace{1cm}+\Big\la {\bf{A}}\Big(\frac{1}{\sqrt{n}}\Phi_0\big(o(e)\big)\Big), v_e\Big\ra_g^2\Bigg)+J_n(\Phi_0, e), 
\end{aligned}
$$
where $J_n(\Phi_0, e)$ satisfies that $|J_n(\Phi_0, e)| \le Cn^{-3/2}$
for some $C=C(\Phi_0)>0$ independent of $e \in E$. 
Denote by $\bm{x}_i$ the $i$th
coefficient of $\bm{x} \in \mathbb{R}^d$ with respect to the Albanese
metric.   
Then, another use of the Taylor formula gives
\begin{align}
n(H_{\frac{1}{n}\omega} - I)P_n^H f
&=-\sqrt{n}\sum_{e \in E_x}p(e)\Bigg\{\sqrt{-1}
\Big\la {\bf{A}}\Big(\frac{1}{\sqrt{n}}\Phi_0(x)\Big), v_e\Big\ra_g
f\Big(\frac{1}{\sqrt{n}}\Phi_0(x)\Big)\nn\\
&\hspace{1cm}+\sum_{i=1}^d \frac{\del f}{\del x_i}\Big(\frac{1}{\sqrt{n}}\Phi_0(x)\Big)(v_e)_i
\Bigg\}\nn\\
&\hspace{1cm}+\frac{1}{2}\sum_{e \in E_x}p(e)
\Bigg\{ \frac{\del^2 f}{\del x_i \del x_j}\Big(\frac{1}{\sqrt{n}}\Phi_0(x)\Big)(v_e)_i(v_e)_j \nn\\
&\hspace{1cm}-2\sqrt{-1}\Big\la {\bf{A}}
\Big(\frac{1}{\sqrt{n}}\Phi_0(x)\Big), v_e\Big\ra_g 
\sum_{i=1}^d \frac{\del f}{\del x_i}\Big(\frac{1}{\sqrt{n}}\Phi_0(x)\Big)(v_e)_i \nn\\
&\hspace{1cm}-\frac{1}{2}
\Big(\sqrt{-1}\la {\bf{A}}v_e, v_e\ra_g+2\sqrt{-1}\pi^*\omega_0(e) \nn\\
&\hspace{1cm}+\Big\la {\bf{A}}\Big(\frac{1}{\sqrt{n}}\Phi_0(x)\Big), v_e\Big\ra_g^2\Big)
f\Big(\frac{1}{\sqrt{n}}\Phi_0(x)\Big)\Bigg\}+\widetilde{J}_n(\Phi_0, x), 
\label{Taylor}
\end{align}
where $\widetilde{J}_n(\Phi_0, f, x)$ satisfies 
$\|\widetilde{J}_n(\Phi_0, f, \cdot)\|_\infty \le Cn^{-1/2}$ for some $C>0$. 
We easily see that the first term of the right-hand side of \eqref{Taylor} is zero since
$$
\sum_{e \in E_x}p(e)v_e=0 \quad \text{and} \quad 
\sum_{e \in E_x}p(e)\omega(e)=0, \qquad x \in V.
$$
As for the second term of the right-hand side of \eqref{Taylor}, 
we can show that it is equal to 
$$
\begin{aligned}
&-\sum_{i=1}^d \frac{\del^2 f}{\del x_i^2}\Big(\frac{1}{\sqrt{n}}\Phi_0(x)\Big)
+2\sqrt{-1}\sum_{i, j=1}^d a_{ij}x_j\frac{\del f}{\del x_i}\Big(\frac{1}{\sqrt{n}}\Phi_0(x)\Big)\\
&\hspace{1cm}+\Big(\sqrt{-1}\sum_{i=1}^d a_{ii}+\sum_{i=1}^d 
\Big(\sum_{j=1}^d a_{ij}x_j\Big)^2\Big)f\Big(\frac{1}{\sqrt{n}}\Phi_0(x)\Big)
+\widetilde{J}'_n(\Phi_0, f, x)\\
&=P_n(\nabla_A^*\nabla_A)f(x)+\widetilde{J}'_n(\Phi_0, f, x)
\end{aligned}
$$
by following the same discussion as \cite[pp. 473 and 474]{Kotani}, 
where $\widetilde{J}'_n(\Phi_0, f, x)$ satisfies 
$\|\widetilde{J}'_n(\Phi_0, f, \cdot)\|_\infty \le Cn^{-1/2}$ for some $C>0$. 
We note that the ergodic theorem for the transition operator 
acting on $\ell^2(X_0)=\{f : V_0 \to \mathbb{C}\}$ plays a crucial role.  
This completes the proof.
\end{proof}

Let $C_x([0, t], \mathbb{R}^d)$ be the set of all continuous functions
$w : [0, t] \to \mathbb{R}^d$ with $w(0)=x\in \mathbb{R}^d$
and $\mu$ the usual Wiener measure on $C_x([0, t], \mathbb{R}^d)$. 
The Schr\"odinger semigroup $(\e^{t\nabla_A^*\nabla_A})_{t \ge 0}$ 
acts on $C_\infty(\mathbb{R}^d)$ and it
is represented as 
$$
(\e^{t \nabla_A^*\nabla_A}f)(x)
=\int_{C_x([0, t], \mathbb{R}^d)}\exp\Big(\sqrt{-1}
\int_0^t A\big(w(s)\big) \circ dw(s)\Big)
f\big(w(t)\big) \, \mu(dw) 
$$
for every $f \in C_\infty(\mathbb{R}^d)$
by virtue of the Feynman--Kac formula, where $\circ dw(s)$ denotes 
the Stratonovich integral. 
Moreover, we verify that  
$\mathrm{Ran}(\lambda - \nabla_A^*\nabla_A)$ is dense
in $C_\infty(\mathbb{R}^d)$ for some $\lambda>0$ and the closure of 
$(\nabla_A^*\nabla_A, C_\infty^\infty(\mathbb{R}^d))$ generates the 
Schr\"odinger semigroup $(\e^{t\nabla_A^*\nabla_A})_{t \ge 0}$ 
(see \cite[Section 1]{Kotani}). 
Since it holds that $(\e^{t\nabla_A^*\nabla_A})(D) \subset D$ and 
$\psi_n(\e^{t\nabla_A^*\nabla_A}f) \le \psi_n(f)$ for $t \ge 0$ and $f \in D$, 
Theorem \ref{Kotani-Trotter} is obtained as a consequence of 
\eqref{Sch-1} and \eqref{Sch-2} in Lemma \ref{Lemma}. 

As far as we know, there seems to be no results establishing the rate of convergence 
of the (generalized) Harper operators to the magnetic Schr\"odinger oeprator. 
Hence, Theorem \ref{Kotani-Trotter} gives a new contribution 
to the study of magnetic Schr\"odinger operators on periodic graphs. 
Since our main result (Theorem \ref{Trotter-refine}) is given 
in full generality, we expect further applications of it in various settings.

\begin{re}
The periodic realization $\Phi_0$ satisfying \eqref{harmonic} is called the 
{\it harmonic realization}, which was introduced in \cite{KS00-TAMS} 
and was regarded as a discrete analogue of harmonic maps on Riemannian manifolds. 
It also describes the most natural configurations of a crystal from a geometric perspective. 
We note that Theorem {\rm \ref{Kotani-Trotter}} as well as 
Proposition {\rm \ref{Kotani-CLT}} holds even when the given realization $\Phi$ is not always 
harmonic, since the difference $|\Phi(x)- \Phi_0(x)|$ is uniformly bounded 
in $x \in V$ due to the periodicities. See also \cite[Section 4]{Kotani} for related discussions. 
\end{re}


\vspace{2mm}
{\bf Acknowledgements}. 
The author would like to thank the anonymous referee for 
providing valuable comments and suggestions. 
This work is supported by JSPS KAKENHI Grant No.\
19K23410.

\end{document}